\documentclass[10pt]{amsart}
\usepackage{amsfonts}
\usepackage{amssymb}
\usepackage{amsmath,amssymb,amsfonts,amsthm,graphics,
latexsym, amscd, amsfonts, epsfig, eepic,epic}
\usepackage{mathrsfs}
\usepackage{algorithm}
\usepackage{algorithmic}
\usepackage{color}
\newcommand{\NEW}[1]{\begingroup\color{red}#1\endgroup}
\newcommand{\TODO}[1]{\begingroup\color{blue}#1\endgroup}
\usepackage{eucal}
\usepackage{eufrak}
\usepackage[all]{xypic}
\usepackage{xspace}

\newtheorem{theorem}{Theorem}[section]

\newtheorem{lemma}[theorem]{Lemma}
\newtheorem{proposition}[theorem]{Proposition}
\newtheorem{corollary}[theorem]{Corollary}

\theoremstyle{remark}

\makeatletter \@addtoreset{equation}{section}

\title{Random $k$-noncrossing partitions}

\author{Jing Qin  and Christian M. Reidys}
\email{duck@santafe.edu}
\date{}
\begin{document}
\maketitle
\begin{abstract}
In this paper, we introduce polynomial time algorithms that generate
random $k$-noncrossing partitions and 2-regular, $k$-noncrossing
partitions with uniform probability. A $k$-noncrossing partition
does not contain any $k$ mutually crossing arcs in its canonical
representation and is $2$-regular if the latter does not contain
arcs of the form $(i,i+1)$. Using a bijection of Chen~{\it et al.}
\cite{Chen,Reidys:08tan}, we interpret $k$-noncrossing partitions
and $2$-regular, $k$-noncrossing partitions as restricted
generalized vacillating tableaux. Furthermore, we interpret the
tableaux as sampling paths of a Markov-processes over shapes and
derive their transition probabilities.
\end{abstract}


\section{Introduction}\label{S:intro}
Recently Chen {\it et.~al.}~\cite{Chen:pnas} studied a computational
paradigm for $k$-noncrossing RNA pseudoknot structures. The authors
regarded $k$-noncrossing matchings as oscillating tableaux
\cite{Chen} and interpreted the latter as stochastic processes over
Young tableaux of less than $k$ rows. According to which, they
generated $k$-noncrossing RNA pseudoknot structures
\cite{Reidys:07pseu} with uniform probability. Since the generating
function of oscillating lattice walks in $\mathbb{Z}^{k-1}$, that
remain in the interior of the dominant Weyl chamber is given as a
determinant of Bessel functions \cite{Grabiner}, it is $D$-finite.
As a result, the transition probabilities of the Markov-process can
be derived in polynomial time as a pre-processing step and each
$k$-noncrossing RNA pseudoknot structure can be generated in linear
time.

The analogue of the above result is for $k$-noncrossing partitions
much more involved. Only for $k=3$, Bousquet-M\'{e}lou and Xin
\cite{Xin:06} compute the generating function and conjecture that
$k$-noncrossing partitions are not $P$-recursive for $k\geq 4$.
\NEW{Since the enumerative result are unavailable for $k\geq 4$, it
is important to construct an algorithm which can uniformly generate
$k$-noncrossing partitions in order to obtain statistic results.}

In this paper, we show that $k$-noncrossing partitions and
$2$-regular, $k$-noncrossing partitions can be uniformly generated
in linear time. We remark that there does not exist a general
framework for the uniform generation of elements of a non-inductive
combinatorial class. In this context relevant work has been done by
Wilf \cite{Wilf:77} as well as \cite{Wilf:book}.

The paper is organized as follows: in Section~\ref{S:basic}, we
present all basic facts on $k$-noncrossing partitions, vacillating
tableaux and lattice walks. In Section~\ref{S:pk}, we compute the
number of lattice walks ending at arbitrary $\nu\in Q_{k-1}$ from
the construction and prove Theorem~\ref{T:algpk}, the uniform
generation result for $k$-noncrossing partitions. In
Section~\ref{S:p2k}, we show how to uniformly generate $2$-regular,
$k$-noncrossing partitions. To this end, we establish a bijection
between lattice walks associated to partitions and braids
\cite{Reidys:08tan}. We will show that walks associated with
$2$-regular partitions correspond to walks associated to loop-free
braids. The latter can easily be dealt with via the
inclusion-exclusion principle. In Theorem~\ref{T:regpart} we present
an algorithm that uniformly generates $2$-regular, $3$-noncrossing
partitions in linear time. \NEW{Finally, in Section~\ref{S:4}, we
show the particular algorithms in case of $k=4$.}

We remark that the generation of random $2$-regular, $k$-noncrossing
partitions has important applications in computational biology. The
latter allow to represent so called base triples in three
dimensional RNA structures, see \cite{Reidys:frame} for details of
the framework. The uniform generation then facilitates energy-based,
{\it ab initio} folding algorithms along the lines of
\cite{Chen:pnas}. This paper is accompanied by supplemental
materials containing MAPLE implementations of the algorithms,
available at ${\tt www.combinatorics.cn/cbpc/par.html}$.

\section{Some basic facts}\label{S:basic}

A \emph{set partition} $P$ of $[N]=\{1,2,\dots, N\}$ is a collection
of nonempty and mutually disjoint subsets of $[N]$, called blocks,
whose union is $[N]$. A $k$-noncrossing partition is called
\emph{$m$-regular}, $m\geq 1$, if for any two distinct elements $x$,
$y$ in the same block, we have $\vert x-y \vert \geq m$. A
\emph{partial matching} and a \emph{matching} is a particular type
of partition having block size at most two and exactly two,
respectively. Their standard representation is a unique graph on the
vertex set $[N]$ whose edge set consists of arcs connecting the
elements of each block in numerical order, see Fig.\ref{F:parcro1}.

\TODO{maybe new similar graph}
\begin{figure}[ht]
\centerline{%
\epsfig{file=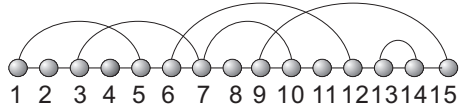,width=0.4\textwidth}\hskip15pt
 }
\caption{\small Standard representation and $k$-crossings: we
display the partition $\{2\}$, $\{8\}$, $\{11\}$, $\{1,5\}$,
$\{6,12\}$, $\{9,15\}$, $\{13,14\}$, $\{3,7,10\}$ of $[15]$.
Elements within blocks are connected in numerical order. The arcs
$\{(1,5),(3,7)\}$, $\{(3,7),(6,12)\}$ $\{(7,10),(9,15)\}$ and
$\{(6,12),(9,15)\}$ are all $2$-crossings. } \label{F:parcro1}
\end{figure}

Given a (set) partition $P$, a \emph{$k$-crossing} is a set of $k$
edges $\{(i_{1},j_{1}),(i_{2},j_{2}),\ldots, (i_{k},j_{k})\}$ such
that $i_{1}< i_{2}< \ldots < i_{k}< j_{1}< j_{2}< \ldots < j_{k}$,
see Fig.\ref{F:parcro1}. A \emph{$k$-noncrossing partition} is a
partition without any $k$-crossings.  We denote the sets of
$k$-noncrossing partitions and $m$-regular, $k$-noncrossing
partitions by $\mathcal{P}_k(N)$ and $\mathcal{P}_{k,m}(N)$,
respectively. For instance, the set of $2$-regular, $3$-noncrossing
partitions are denoted by $\mathcal{P}_{3,2}(N)$.

A \emph{(generalized)} vacillating tableau \cite{Reidys:08tan}
$V_{\lambda}^{2N}$ of shape $\lambda$ and length $2N$ is a sequence
$\lambda^{0}, \lambda^{1},\ldots,\lambda^{2N}$ of shapes such that
{\sf (1)} $\lambda^{0}=\varnothing$, $\lambda^{2N}=\lambda$ and {\sf
(2)} for $1\le i\le N$, $\lambda^{2i-1},\lambda^{2i}$ are derived
from $\lambda^{2i-2}$ by \emph{elementary moves} (EM) defined as
follows: $(\varnothing,\varnothing)$: do nothing twice;
$(-\square,\varnothing)$: first remove a square then do nothing;
$(\varnothing,+\square)$: first do nothing then add a square; $(\pm
\square,\pm \square)$: add/remove a square at the odd and even
steps, respectively. We use the following notation: if
$\lambda_{i+1}$ is obtained from $\lambda_{i}$ by adding, removing a
square from the $j$-th row, or doing nothing we write
$\lambda_{i+1}\setminus\lambda_{i}=+\square_{j}$,
$\lambda_{i+1}\setminus\lambda_{i}= -\square_{j}$ or
$\lambda_{i+1}\setminus\lambda_{i}=\varnothing$, respectively, see
Fig.\ref{F:tableaux}.

\begin{figure}[ht]
\centerline{\includegraphics[width=0.75\textwidth]{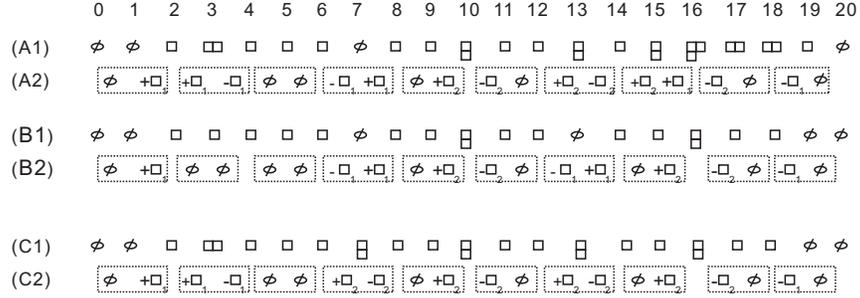}}
\caption{\small
 Vacillating tableaux and elementary moves:
\textsf{(A1)} shows a general vacillating tableaux with EM
\textsf{(A2)}. \textsf{(B1)} displays another vacillating tableaux
and its set of EM, $\{(-\square,\varnothing),
(\varnothing,+\square),(\varnothing,\varnothing),(-\square,+\square)\}$
shown in \textsf{(B2)}. In \textsf{(C1)} we present a vacillating
tableaux with the EM $\{(-\square,\varnothing),
(\varnothing,+\square),(\varnothing,\varnothing),(+\square,-\square)\}$
displayed in \textsf(C2).} \label{F:tableaux}
\end{figure}

A braid over $[N]$ can be represented via introducing loops $(i,i)$
and drawing arcs $(i,j)$ and $(j,\ell)$ with $i<j<\ell$ as crossing,
see Fig.\ref{F:braid} \textsf{(B1)}. A \emph{$k$-noncrossing braid}
is a braid without any $k$-crossings. We denote the set of
$k$-noncrossing braids over $[N]$ with and without isolated points
by $\mathcal{B}_k(N)$ and $\mathcal{B}^{*}_k(N)$, respectively. Chen
{\it et al.} \cite{Chen} have shown that each $k$-noncrossing
partition corresponds uniquely to a vacillating tableau of empty
shape, having at most $(k-1)$ rows, obtained via the EMs
$\{(-\square, \varnothing),(\varnothing, +\square), (\varnothing,
\varnothing),(-\square, +\square)\}$, see Fig.\ref{F:braid}
\textsf{(A2)}. In \cite{Reidys:08tan}, Chen {\it et al.} proceed by
proving that vacillating tableaux of empty shape, having at most
$(k-1)$-rows which are obtained by the EMs $\{(-\square,
\varnothing),(\varnothing, +\square),(\varnothing, \varnothing),
(+\square, -\square)\}$ correspond uniquely to $k$-noncrossing
\emph{braids}.
\begin{figure}[ht]
\centerline{%
\epsfig{file=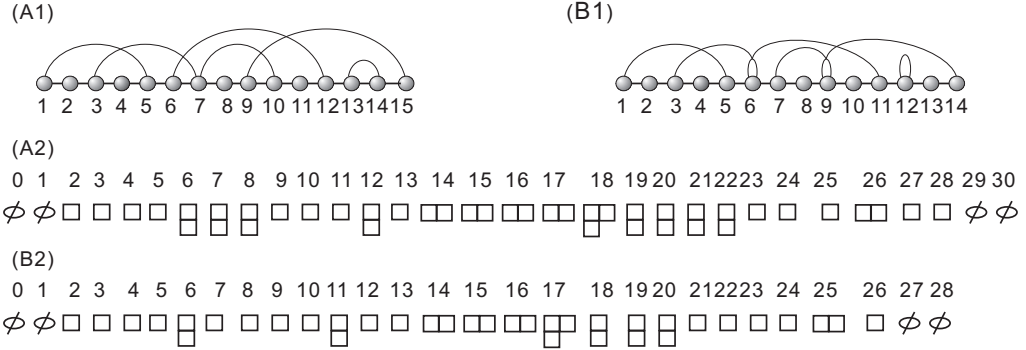,width=0.9\textwidth}\hskip15pt
 }
\caption{\small Vacillating tableaux, partitions and braids: in
\textsf{(A1)} we show a $3$-noncrossing partition and in
\textsf{(A2)} the associated vacillating tableau. \textsf{(B1)}
shows a $3$-noncrossing braid and  \textsf{(B2)} its vacillating
tableau.} \label{F:braid}
\end{figure}

In the following, we consider $k$-noncrossing partition or a
$2$-regular, $k$-noncrossing partition of $N$ vertices, where $N
\geq 1$. To this end, we introduce two $\mathbb{Z}_{k-1}$ domains
$Q_{k-1}=\{\nu=(\nu_1,\nu_2,\dots,\nu_{k-1})\in\mathbb{Z}_{k-1}\mid
\nu_1, \nu_2,\dots,\nu_{k-1}\geq 0\}$ and $W_{k-1}=
\{\nu=(\nu_1,\nu_2,\dots,\nu_{k-1})\in\mathbb{Z}_{k-1}\mid
\nu_1>\nu_2>\dots>\nu_{k-1}\geq 0\}$. Set $D_{k-1}\in
\{Q_{k-1},W_{k-1}\}$, $\mathbf{e_i}$ be the unit of $i$-th axis,
where $i\in \{1,2,\dots, k-1\}$ and $\mathbf{e_{0}}=\mathbf{0}$ be
the origin. A \emph{$\mathcal{P}_{D_{k-1}}$-walk} is a lattice walk
in $D_{k-1}$ starting at $\epsilon=(k-2,,k-1,\dots,0)$ having steps
$\pm \mathbf{e_{i}}$, such that even steps are $+\mathbf{e_{i}}$ and
odd steps are either $-\mathbf{e_{i}}$, where
$i\in\{0,1,\dots,k-1\}$.
Analogously, a \emph{$\mathcal{B}_{D_{k-1}}$-walk} is a lattice walk
in $D_{k-1}$ starting at $\epsilon$ whose even steps are either
$-\mathbf{e_{i}}$ and whose odd steps are $+\mathbf{e_{i}}$, where
$i\in\{0,1,\dots,k-1\}$. By abuse of language, we will omit the
subscript $D_{k-1}$.

Interpreting the number of squares in the rows of the shapes as
coordinates of lattice points, we immediately obtain
\begin{theorem}\label{T:bijection}\cite{Reidys:08tan}\\
{\rm (1)} The number of $k$-noncrossing partitions over $[N]$ equals
the number of $\mathcal{P}_{W_{k-1}}$-walks from $\epsilon$ to
itself of length $2N$. \\
{\rm (2)} The number of $k$-noncrossing braids over $[N]$ equals the
number of $\mathcal{B}_{W_{k-1}}$-walks from $\epsilon$ to itself of
length $2N$.
\end{theorem}
Recall $\delta=(k-2,,k-1,\dots,0)$, Bousquet-M\'{e}lou and Xin
\cite{Xin:06} according to the reflection principle, arrive at the
following proposition:
\begin{proposition}\label{P:reflect}
For any ending points $\mu\in W_{k-1}$, set $\omega^{\mu}_{k,\ell}$
and $a^{\mu}_{k,\ell}$ denote the number of $W_{k-1}$-walks and
$Q_{k-1}$ going from $\delta$ to $\mu$ of length $\ell$,
respectively. Then we have
\begin{equation}
\omega^{\mu}_{k,\ell}=\sum_{\pi\in \mathbb{S}_{k-1}}{\sf
sgn}(\pi)a^{\pi(\mu)}_{k,\ell},
\end{equation}
where $\mathbb{S}_{k-1}$ is the permutation group of $[k-1]$, ${\sf
sgn}(\pi)$ is the sign of $\pi$ and
\begin{equation}
\pi(\mu)=\pi(\mu_1,\mu_2,\dots,\mu_{k-1})=(\mu_{\pi(1)},\mu_{\pi(2)},\dots,
\mu_{\pi(k-1)}).
\end{equation}
\end{proposition}
\section{random $k$-noncrossing partitions}\label{S:pk}
In this section, we establish an algorithm to uniformly generate a
$k$-noncrossing partition of $[N]$ employing a Markov-process as an
interpretation of a vacillating tableau.
\NEW{
\begin{lemma}\label{L:pk}
{\rm (a)} Suppose  $1\leq \ell\leq N-1$ and
$\nu=(\nu_1,\nu_2,\dots,\nu_{k-1})\in Q_{k-1}$, then set
$a^{\nu}_{k,2\ell+1}=f^{\nu}_{k,{\ell}}$, accordingly we have
\begin{equation}\label{E:lk0}
a^{\nu}_{k,s}=
\begin{cases}
f^{\nu}_{k,\ell}& \text{\rm for $s=2\ell+1$},\\
\sum_{i=0}^{k-1}f^{\nu-\mathbf{e_{i}}}_{k, {\ell}} &\text{\rm for $s=2\ell+2$}.\\
\end{cases}
\end{equation}
where $f^{\nu}_{k,\ell}=0$ for $\sum_{i}\nu_{i}\geq
\ell+\frac{(k-1)(k-2)}{2}$ and for $\ell=0$,
\begin{equation*}
f^{\nu}_{k,0}=
\begin{cases}
1 & \textrm{for } \nu=\lambda; \\
0 & \textrm{otherwise.}
\end{cases}
\end{equation*}
{\rm (b)} $f^{\nu}_{k,\ell}$ satisfies the recursion
\begin{equation}\label{E:rk}
f^{\nu}_{k,\ell}=
\sum_{i=0}^{k-1}\sum_{i=0}^{k-1}f^{\nu+\mathbf{e_{i}}-
\mathbf{e_{j}}}_{k,\ell-1}.
\end{equation}
\end{lemma}
}
\begin{proof}
Assertion {\rm (a)} is obvious by construction. To prove assertion
{\rm (b)}, indeed by definition of $f^{\nu}_{k,\ell}$, we have
$a^{\nu}_{k,{2\ell+1}}= f^{\nu}_{k,{\ell}}$. Furthermore, since each
lattice walk ending at $\nu$ of length $2\ell+1$ can be obtained
from a lattice walk ending at
$(\nu-\mathbf{e_{i}}+\mathbf{e_{j}})\in W_{k-1}$ of length $2\ell-1$
via first $+\mathbf{e_{i}}$ then $-\mathbf{e_{j}}$. I.e.~we obtain
eq.~\ref{E:rk} and the proof of the lemma is complete.
\end{proof}

In the following, we consider a vacillating tableau as the sampling
path of a Markov-process, whose transition probabilities can be
calculated via the terms $\omega^{\nu}_{k,\ell}$.

\begin{theorem}\label{T:algpk}
A random $k$-noncrossing partition can be generated, in $O(N^k)$
pre-processing time and $O(N^{k})$ space complexity, with uniform
probability in linear time. Each $k$-noncrossing partition is
generated with $O(N)$ space and time complexity, see Algorithm 1.
\end{theorem}
\begin{algorithm}\label{A:alg1}
\begin{algorithmic}[1]
\STATE{$i$=1}

\STATE{\it Tableaux} (Initialize the sequence of shapes,
$\{\lambda^{i}\}_{i=0}^{i=2N}$)

\STATE {$\lambda^{0}=\varnothing$, $\lambda^{2N}=\varnothing$}

\WHILE {$i < 2N$ }

\IF {$i$ is even} \STATE  X[0]$\leftarrow$ ${\sf
V_{k}}(\lambda^{i+1}_{\varnothing}, 2N-(i+1))$

\FOR {$1\leq j\leq k-1$}

\STATE  X[j]$\leftarrow$ ${\sf V_{k}}(\lambda^{i+1}_{-\square_{j}},
2N-(i+1))$

\ENDFOR

\ENDIF

\IF {$i$ is odd}

\STATE  X[0]$\leftarrow$ ${\sf V_{k}}(\lambda^{i+1}_{\varnothing},
2N-(i+1))$

\FOR {$1\leq j\leq k-1$}

\STATE  X[j]$\leftarrow$ ${\sf V_{k}}(\lambda^{i+1}_{+\square_{j}},
2N-(i+1))$

\ENDFOR

\ENDIF

\STATE {\it sum} $\leftarrow$ ($\sum_{j=1}^{j=k-1}$X[j])+X[0]

\STATE {\it Shape} $\leftarrow$ {\sf Random}({\it sum}) ({\sf
Random} generates the random shape $\lambda^{i+1}_{+\square_{j}(or
-\square_{j})}$ with probability ${X[j]}/{\it sum}$ or
$\lambda^{i+1}_{\varnothing}$ with probability $X[0]/sum$)

\STATE $i \leftarrow i+1$

\STATE Insert {\it Shape} into {\it Tableau} (the sequence of
shapes).

\ENDWHILE

\STATE  {\sf Map}({\it Tableau}) (maps {\it Tableau} into its
corresponding $k$-noncrossing partition)
\end{algorithmic}
\caption{\small Uniform generation of $k$-noncrossing partitions}
\end{algorithm}
\begin{proof}
The main idea is to interpret tableaux of $k$-noncrossing partitions
as sampling paths of a stochastic process, see Fig.\ref{F:rest} for
the case $k=3$. We label the $(i+1)$-th shape,
$\lambda^{i+1}_{\alpha}$ by $\alpha=\lambda^{i+1}\setminus
\lambda^{i}$ and $\alpha\in \{+\square_{j}, -\square_{j},
\varnothing\}_{j=1}^{j=k-1}$, specifying the particular transition
from $\lambda^i$ to
$\lambda^{i+1}$.\\
Let ${\sf V}_{k}(\lambda^{i+1}_{\alpha},2N-(i+1))$ denote the number
of vacillating tableaux of length $(i+1)$ such that
$\lambda^{i+1}\setminus \lambda^{i}=\alpha$. We remark here that if
$\lambda^{i+1}$ has $(\nu_{i}-k-i+3)$ boxes in the $i$-th row, then
${\sf V}_{k}(\lambda^{i+1}_{\alpha},2N-(i+1))=
\omega^{\nu}_{k,2N-(i+1)}$, where
$\nu=(\nu_1,\nu_2,\dots,\nu_{k-1})$.
\begin{figure}[ht]
\centerline{\includegraphics[width=0.6
\textwidth]{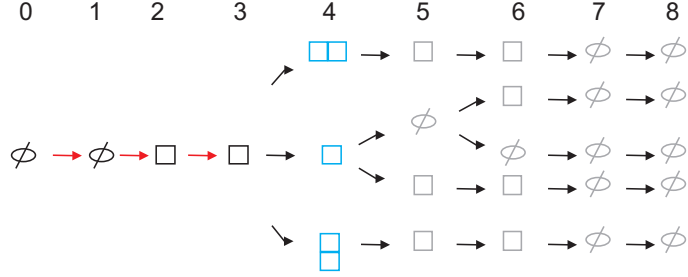}} \caption{\small Sampling paths: we
display all tableaux sequences where $\lambda^1=\varnothing$,
$\lambda^2=\lambda^3=\square$. We display the three possible
$\lambda^4$-shapes (blue), induced by $\varnothing,+
\square_1,+\square_2$ and all possible continuations (grey).}
\label{F:rest}
\end{figure}
Let $(X^i)_{i=0}^{2N}$ be given as follows:\\
$\bullet$ $X^0=X^{2N}=\varnothing$ and $X^i$ is a shape having at
most $k-1$ rows;\\
$\bullet$ for $1\le i\le N-1$, we have $X_{2i+1}\setminus X_{2i}\in
\{\varnothing, -\square_{j}\}_{j=1}^{j=k-1}$ and $X_{2i+2}\setminus
X_{2i+1}\in
\{\varnothing, +\square_{j}\}_{j=1}^{j=k-1}$.\\
$\bullet$ for $1\leq i\leq 2N-1$, we have
\begin{equation}\label{E:w0}
\mathbb{P}(X^{i+1}=\lambda^{i+1}\mid X^i=\lambda^i)= \frac{{\sf
V}(\lambda^{i+1}_{\alpha},2N-i-1)}{{\sf V_{k}}(\lambda^{i},2N-i)}.
\end{equation}
In view of eq.~(\ref{E:w0}), we immediately observe
\begin{equation}
\prod_{i=0}^{2N-1}\mathbb{P}(X^{i+1}=\lambda^{i+1}\mid
X^i=\lambda^i)= \frac{{\sf V_{k}}(\lambda^{2N}=\varnothing,0)} {{\sf
V_{k}}(\lambda^{0}=\varnothing,N)} =\frac{1}{{\sf
V_{k}}(\varnothing,2N)}.
\end{equation}
Consequently, the process $(X^i)_{i=0}^{2N}$ generates random
$k$-noncrossing partitions with uniform probability in $O(N)$ time and space.\\
As for the derivation of the transition probabilities, suppose we
are given a shape $\lambda^{h}$ having exactly $(\nu_{i}-k-i+3)$
boxes in the $i$-th row, where $i\in\{1,2,\dots,k-1\}$. According to
Lemma~\ref{L:pk}, set $\nu=(\nu_1,\nu_2,\dots,\nu_{k-1})$, then we
reduce the problem obtain $\omega^{\nu}_{\ell}$ to the calculation
of several coefficients $f^{\beta}_{\ell}$ for some fixed $\beta\in
W_{k-1}$ related to $\nu$
and $1\leq \ell\leq N-1$. \\
Furthermore, according to Lemma~\ref{L:pk} {\rm (b)}, the
coefficients $f^{\beta}_{\ell}$ for all $\beta\in W_{k-1}$ and
$1\leq \ell\leq N-1$, can be computed in $O(N^k)$ time and $O(N^k)$
space complexity via $k$-nested {\bf
For}-loops.\\
Given the coefficients $f^{\beta}_{\ell}$ for all $\beta,\ell$, we
can derive the transition probabilities in $O(1)$ time. Accordingly,
we obtain all transition probabilities ${\sf
V_{k}}(\lambda^i_{\alpha},2N-i))$ in $O(N^k)$ time and $O(N^{k})$
space complexity.
\end{proof}

\section{random $2$-regular, $k$-noncrossing partitions}\label{S:p2k}
In this section, we generate random $2$-regular, $k$-noncrossing
partitions employing a bijection between the set of $2$-regular,
$k$-noncrossing partitions over $[N]$ and the set of $k$-noncrossing
braids without loops over $[N-1]$.

\begin{lemma}\label{L:bijection-2}
Set $k\in\mathbb{N}$, $k\ge 3$. Suppose
$(\lambda^i)_{i=0}^{2\ell+1}$ in which $\lambda^{0}=\varnothing$ is
a sequence of shapes such that
$\lambda^{2j}\setminus\lambda^{2j-1}\in
\{\{+\square_{h}\}_{h=1}^{h=k-1},\varnothing\}$ and
$\lambda^{2j-1}\setminus \lambda^{2j-2}\in
\{\{-\square_{h}\}_{h=1}^{h=k-1},\varnothing\}$. Then
$(\lambda^i)_{i=0}^{2\ell+1}$ induces a unique sequence of shapes
$(\mu^i)_{i=0}^{2\ell}$ with the following properties
\begin{eqnarray}
&& \label{E:right}
 \mu^{2j+1}\setminus
\mu^{2j}\in \{\{+\square_{h}\}_{h=1}^{h=k-1},\varnothing\}\
\text{\it and}\ \mu^{2j+2}\setminus\mu^{2j+1}\in
\{\{-\square_{h}\}_{h=1}
^{h=k-1},\varnothing\},\\
\label{E:shape} &&
 \mu^{2j}=\lambda^{2j+1}\\\label{E:shape1}
&&  \mu^{2j+1}\neq\lambda^{2j+2}\quad \Longrightarrow\quad
\mu^{2j+1}\in \{\lambda^{2j+1},\lambda^{2j+3}\} .
\end{eqnarray}
\end{lemma}
\begin{figure}[ht]
\centerline{%
\epsfig{file=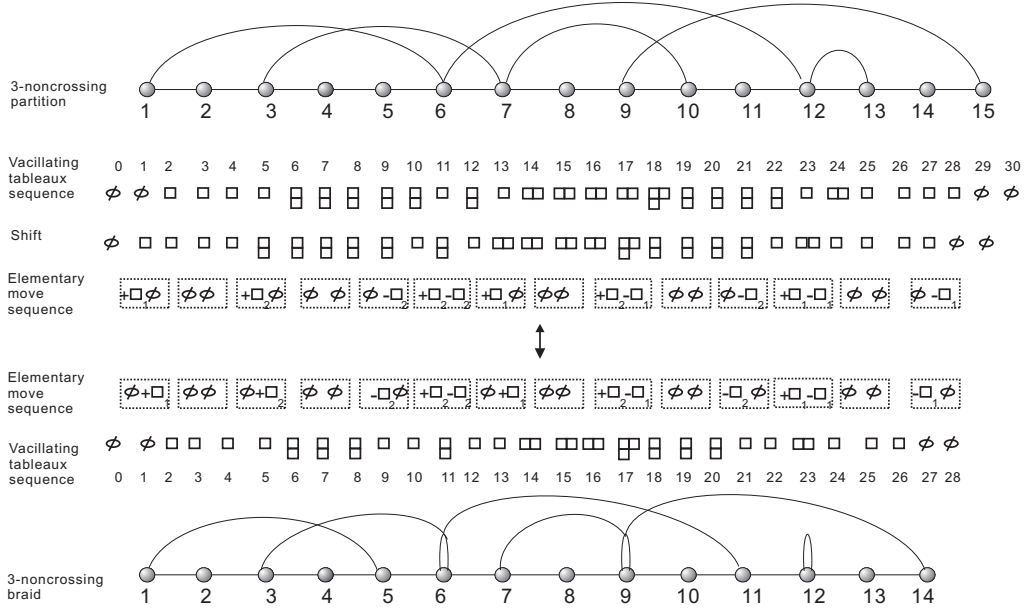,width=0.9\textwidth}\hskip15pt
 }
\caption{\small Illustration of Lemma~\ref{L:bijection-2}. We show
how to map the $\mathcal{P}_{W_2}$-walk induced by the
$3$-noncrossing partition $P$, of Fig.\ref{F:braid}, \textsf{(A)}
into the $\mathcal{B}_{W_2}$-walk. The latter corresponds to the
$3$-noncrossing braid of Fig.\ref{F:braid}, \textsf{(B)}. Here each
vertex, $i$, is aligned with the triple of shapes
$(\lambda^{2i-2},\lambda^{2i-1},\lambda^{2i})$ in the corresponding
tableaux.} \label{F:map}
\end{figure}
\begin{proof}
Since $\lambda^1=\varnothing$, $(\lambda^j)_{j=1}^{2\ell+1}$
corresponds to a sequence of pairs $((x_i,y_i))_{i=1}^{\ell}$ given
by $x_{i}=\lambda^{2i}\setminus\lambda^{2i-1}$ and
$y_{i}=\lambda^{2i+1}\setminus\lambda^{2i}$ such that
\begin{equation}\label{E:neu}
\forall\, 1\le i\le \ell;\qquad (x_i,y_i)\in
\{(\varnothing,\varnothing),\  (+\square_h,\varnothing), \
(\varnothing,-\square_h), \ (+\square_h,-\square_j)\}.
\end{equation}
Let $\varphi$ be given by
\begin{equation}\label{E:varphi2}
\varphi(({x}_i,{y}_i))=
\begin{cases}
({x}_i,{y}_i) & \ \text{\rm for } \
({x}_i,{y}_i)=(+\square_h,-\square_j) \\
({y}_i,{x}_i) & \ \text{\rm otherwise,}
\end{cases}
\end{equation}
and set $(\varphi(x_i,y_i))_{i=1}^\ell=(a_i,b_i)_{i=1}^{\ell}$. Note
that $(a_i,b_i)\in
\{(-\square_h,\varnothing),(\varnothing,+\square_h),
(\varnothing,\varnothing),(+\square_h,-\square_j)\}$, where $1\leq
h,j\leq k-1$, see Fig.~\ref{F:map}.
 Let $(\mu^i)_{i=1}^{\ell}$ be the sequence of shapes
induced by $(a_i,b_i)_{i=1}^{\ell}$ according to
$a_{i}=\mu^{2i}\setminus\mu^{2i-1}$ and
$b_{i}=\mu^{2i-1}\setminus\mu^{2i-2}$ initialized with
$\mu_{0}=\varnothing$. Now eq.~(\ref{E:right}) is implied by
eq.~(\ref{E:neu}) and eq.~(\ref{E:shape}) follows by construction.
Suppose $\mu^{2j+1}\neq\lambda^{2j+2}$ for some $0\le j\le \ell-1$.
By definition of $\varphi$, only pairs containing ``$\varnothing$''
in at least one coordinate are transposed from which we can conclude
$\mu^{2j+1}=\mu^{2j}$ or $\mu^{2j+1}=\mu^{2j+2}$, whence
eq.~(\ref{E:shape1}). I.e.~we have the following situation
\begin{equation*}
\diagram & \lambda^{2j+1} \ar@{=}[dl]\rto^{} &
\lambda^{2j+2}\ar@{-}[dl] \rto^{} &
\lambda^{2j+3} \ar@{=}[dl]  \\
 \mu^{2j} \ar@{=}[r] & \mu^{2j+1} \rto^{} &\mu^{2j+2}
\enddiagram
 \text{\rm or }
\diagram & \lambda^{2j+1} \ar@{=}[dl]\rto^{} &
\lambda^{2j+2}\ar@{-}[dl] \rto^{} &
\lambda^{2j+3} \ar@{=}[dl]  \\
 \mu^{2j} \rto^{} & \mu^{2j+1} \ar@{=}[r] &\mu^{2j+2},
\enddiagram
\end{equation*}
and the lemma follows.
\end{proof}

Lemma~\ref{L:bijection-2} establishes a bijection between
$\mathcal{P}_{W_{k-1}}$-walks of length $2\ell+1$ and
$\mathcal{B}_{W_{k-1}}$-walks of length $2\ell$, where
$W_{k-1}=\{(a_1,a_2,\dots,a_{k-1})\mid a_1>a_2>\dots>a_{k-1}\}$.
\begin{corollary}\label{C:bi}
Let $\mathcal{P}_k(N)$ and $\mathcal{B}_{k}(N-1)$ denote the set of
$k$-noncrossing partitions on $[N]$ and $k$-noncrossing braids on
$[N-1]$.
Then \\
{\rm (a)} we have a bijection
\begin{equation}\label{E:biject2}
\vartheta\colon \mathcal{P}_{k}(N)\longrightarrow
\mathcal{B}_k(N-1).
\end{equation}
{\rm (b)} $\vartheta$ induces by restriction a bijection between
$2$-regular, $k$-noncrossing partitions on $[N]$ and $k$-noncrossing
braids without loops on $[N-1]$.
\end{corollary}

\begin{figure}[ht]
\centerline{%
\epsfig{file=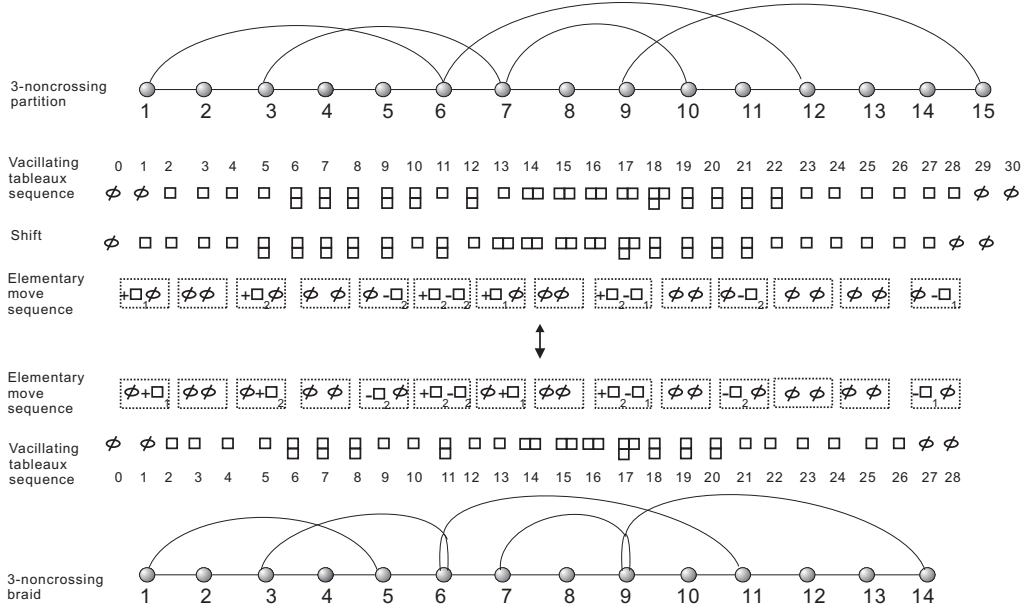,width=0.9\textwidth}\hskip15pt
 }
\caption{\small Mapping $2$-regular, $3$-noncrossing partitions into
loop-free braids: illustration of Corollary~\ref{C:bi}(b).}
\label{F:map1}
\end{figure}
\begin{proof}
Assertion (a) follows from Lemma~\ref{L:bijection-2} since a
partition is completely determined by its induced
$\mathcal{P}_{W_{k-1}}$-walk of shape $\lambda^{2n-1}=\varnothing$.
(b) follows immediately from the fact that, according to the
definition of $\varphi$ given in Lemma~\ref{L:bijection-2}, any pair
of consecutive EMs $(\varnothing,
+\square_1),(-\square_1,\varnothing)$ induces an EM
$(+\square_1,-\square_1)$. Therefore, $\vartheta$ maps $2$-regular,
$k$-noncrossing partitions into $k$-noncrossing braids without
loops. We illustrate Corollary~\ref{C:bi}(b) in Fig.\ref{F:map1}.
\end{proof}
\begin{lemma}\label{L:regq}
The number of $\mathcal{B}_{W_{k-1}}^*$-walks ending at $\nu\in
W_{k-1}$ of length $2\ell$, where $1\leq \ell\leq N$ is given by
\begin{equation}
\sigma^{\nu;*}_{k,2\ell} = \sum_{h}(-1)^{h}{\ell\choose
h}\omega^{\nu}_{2(\ell-h)+1}.
\end{equation}
Furthermore, we obtain the recurrence of
$\sigma^{\nu;*}_{k,2\ell-1}$ given by
\begin{equation}
\sigma^{\nu;*}_{k,2\ell-1}
=\sum_{j=0}^{k-1}\sigma^{\nu-\mathbf{e_{j}};*}_{k,2\ell+2}.
\end{equation}
\end{lemma}
\begin{proof}
According to Lemma~\ref{L:bijection-2}, any
$\mathcal{P}_{W_{k-1}}$-walk of length $2\ell+1$ corresponds to a
unique $\mathcal{B}_{W_{k-1}}$-walk of length $2\ell$, whence
$\sigma^{\nu}_{k,2\ell}=\omega^{\nu}_{k,2\ell+1}$. Let
$A_{2\ell}(h)$ denote the set of $\mathcal{B}_{W_{k-1}}$-walks of
length $2\ell$ in which there exist at least $h$ pairs of shapes
$(\mu_{2q-1},\mu_{2q})$ induced by the EM $(+\mathbf{e_{1}},
-\mathbf{e_{1}})$. Since the removal of $h$ EMs, $(+\mathbf{e_{1}},
-\mathbf{e_{1}})$, from a $\mathcal{B}_{W_{k-1}}$-walk, results in a
$\mathcal{B}_{W_{k-1}}$-walk of length $2(\ell-h)$, we derive
$A_{2\ell}(h)= {\ell \choose h}\,\sigma^{i,j}_{k,2\ell-2h}$. Using
the inclusion-exclusion principle, we arrive at
\begin{equation}\label{E:IE}
\sigma^{\nu;*}_{k,2\ell}=\sum_{h}(-1)^{h}{\ell \choose
h}\sigma^{\nu}_{k,2(\ell-h)}= \sum_{h}(-1)^{h}{\ell \choose
h}\omega^{\nu}_{k,2(\ell-h)+1}.
\end{equation}
By construction, an odd step in a $\mathcal{B}_{W_{k-1}}$-walk is
$+\mathbf{e_{j}}$, where $j\in\{0,1,\dots,k-1\}$, whence
\begin{equation*}
\sigma^{\nu;*}_{k,2\ell-1}
=\sum_{j=0}^{k-1}\sigma^{\nu-\mathbf{e_{j}};*}_{k,2\ell-2}.
\end{equation*}
\end{proof}
Via Lemma~\ref{L:regq}, we have explicit knowledge about the numbers
of $\mathcal{B}_{W_{k-1}}^*$-walks, i.e.~$\sigma^{\nu;*}_{k,s}$ for
all $\nu\in W_{k-1}$ and $1\leq s\leq 2N$. Accordingly, we are now
in position to generate $2$-regular, $k$-noncrossing partitions with
uniform probability via $k$-noncrossing, loop-free braids.
\begin{theorem}\label{T:regpart}
A random $2$-regular, $k$-noncrossing partition can be generated, in
$O(N^{k+1})$ pre-processing time and $O(N^{k})$ space complexity,
with uniform probability in linear time. Each $2$-regular,
$k$-noncrossing partition is generated with $O(N)$ space and time
complexity, see Algorithm 2, 3.
\end{theorem}
\begin{algorithm}
\label{A:alg2}
\begin{algorithmic}[1]
\STATE{\it Tableaux} (Initialize the sequence of shapes to be a list
$\{\lambda^{i}\}_{i=0}^{i=2N}$)

\STATE {$\lambda^{0}\leftarrow\varnothing$,
$\lambda^{2N}\leftarrow\varnothing$,
$\lambda^{2N-1}\leftarrow\varnothing$, $i\leftarrow1$}

\WHILE {$i < 2N-1$ }

\IF {$i$ is even}

\STATE  X[0]$\leftarrow$ ${\sf
V_{k}^{*}}(\lambda^{i+1}_{\varnothing}, 2N-(i+1))$

\FOR{$j=1$ to $k-1$}

\STATE X[$2j-1$]$\leftarrow$ ${\sf
V_{k}^*}(\lambda^{i+1}_{+\square_{j}}, 2N-(i+1))$,
X[$2j$]$\leftarrow$ ${\sf V_{k}^*}(\lambda^{i+1}_{-\square_{j}},
2N-(i+1))$

\ENDFOR

\ELSE

\STATE  X[0]$\leftarrow$ ${\sf W^{*}}(\lambda^{i+1}_{\varnothing},
2N-(i+1))$

\FOR{$j=1$ to $k-1$}

\STATE X[$2j-1$]$\leftarrow$ ${\sf
W_{k}^*}(\lambda^{i+1}_{+\square_{j}}, 2N-(i+1))$,
X[$2j$]$\leftarrow$ ${\sf W_{k}^*}(\lambda^{i+1}_{-\square_{j}},
2N-(i+1))$

\ENDFOR

\STATE {\sf FLAG}(flag0, flag1, flag2, flag3)

\STATE {\it sum} $\leftarrow$ $\sum_{t=0}^{2(k-1)}$X[t]

\ENDIF

\STATE {\it Shape} $\leftarrow$ {\sf Random}({\it sum}) ({\sf
Random} generates the random shape $\lambda^{i+1}_{+\square_{j}(or
-\square_{j})}$ with probability {X[2j-1]}$/{\it sum}$ (or
X[2j]$/sum$) or $\lambda^{i+1}_{\varnothing}$ with probability
X[0]$/sum$)

\STATE {flag0 $\leftarrow$ 1, flag1 $\leftarrow$ 1, flag2
$\leftarrow$ 1, flag3 $\leftarrow$ 1}

\IF {$i$ is even and Shape=$\lambda^{i+1}_{+\square_{1}}$}

\STATE flag1 $\leftarrow$ 0

\ENDIF

\FOR {$1\leq s\leq k-1$}

\IF {$i$ is even and Shape=$\lambda^{i+1}_{-\square_{s}}$}

\STATE flag2 $\leftarrow$ 0

\ENDIF

\ENDFOR

\FOR {$2\leq s\leq k-1$}

\IF {$i$ is even and Shape=$\lambda^{i+1}_{+\square_{s}}$ }

\STATE flag3 $\leftarrow$ 0

\ENDIF

\ENDFOR

\IF {$i$ is even and Shape=$\lambda^{i+1}_{\varnothing}$}

\STATE flag0 $\leftarrow 0$

\ENDIF

\STATE Insert {\it Shape} into {\it Tableaux} (the sequence of
shapes)

\STATE $i \leftarrow i+1$

\ENDWHILE

\STATE  {\sf Map}({\it Tableaux}) (maps {\it Tableaux} into its
corresponding $2$-regular, $k$-noncrossing partition)
\end{algorithmic}
\caption{\small {\sf Tableaux}: Uniform generation of $2$-regular,
$k$-noncrossing partitions.}
\end{algorithm}
\begin{algorithm}
\label{A:alg3}
\begin{algorithmic}[1]
\STATE{\it FLAG}(flag0, flag1, flag2, flag3)

\IF {flag0 =0}

\FOR{$j=1$ to $k-1$}

\STATE X[$2j$]$\leftarrow 0$

\ENDFOR

\ENDIF

\IF {flag1 =0}

\STATE X[$0$]$\leftarrow 0$, X[$2$]$\leftarrow 0$

\FOR{$j=1$ to $k-1$}

\STATE X[$2j-1$]$\leftarrow 0$

\ENDFOR

\ENDIF

\IF {flag2 =0}

\FOR{$j=1$ to $k-1$}

\STATE X[$2j-1$]$\leftarrow 0$, X[$2j$]$\leftarrow 0$

\ENDFOR

\ENDIF

\IF {flag3 =0}

\FOR{$j=1$ to $k-1$}

\STATE X[$2j-1$]$\leftarrow 0$

\ENDFOR

\ENDIF

\end{algorithmic}
\caption{\small {\sf FLAG}: Distinguish the last odd step}
\end{algorithm}
\begin{proof}
We interpret $\mathcal{B}_{W_{k-1}}^*$-walks as sampling paths of a
stochastic process. To this end, we again label the $(i+1)$-th
shape, $\lambda^{i+1}_{\alpha}$ by $\alpha=\lambda^{i+1}\setminus
\lambda^{i}$,
$\alpha\in\{\{+\square_{j}\}_{j=1}^{j=k-1},\{-\square_{j}\}_{j=1}^{j=k-1},
\varnothing\}$,
where the labeling specifies the transition from $\lambda^i$ to
$\lambda^{i+1}$. In the following, we distinguish even and odd
labeled shapes. \\
For $i=2s$, we let $W_{k}^{*}(\lambda^{2s}_\alpha, 2N-2s)$ denote
the number of $\mathcal{B}_{W_{k-1}}^*$-walks such that
$\alpha=\lambda^{2s}\setminus \lambda^{2s-1}$. By construction,
assume $\lambda^{2s}$ has $\nu_{i}-k-i+3$ boxes in $i$-th row,
respectively, then we have
\begin{equation}\label{E:ws}
W_{k}^{*}(\lambda^{2s}_\alpha, 2N-2s)=\sigma^{\nu;*}_{k,2N-2s}
\end{equation}
i.e.~$W^{*}_{k}(\lambda^{2s}_\alpha, 2N-2s)$  is independent of
$\alpha$ and we write $W_{k}^{*}(\lambda^{2s}_\alpha, 2N-2s)=
W_{k}^{*}(\lambda^{2s}, 2N-2s)$. \\
For $i=2s+1$, let $V_{k}^{*}(\lambda^{2s+1}_{\alpha},2N-(2s+1))$
denote the number of $\mathcal{B}_{W_{k-1}}^*$-walks of shape
$\lambda^{2s+1}$ of length $2N-(2s+1)$ where
$\lambda^{2s+1}\setminus\lambda^{2s}=\alpha$. Then we have setting
$u=2N-2s-2$
\begin{equation}
V_{k}^{*}(\lambda^{2s+1}_{\alpha},u+1)=
\begin{cases}
\sum_{j=2}^{k-1}W_{k}^{*} (\lambda^{2s+2}_{-\square_{j}},u) &
\text{\rm for} \
\alpha=+\square_{1};\\
\sum_{t\neq
0,j}W_{k}^{*}(\lambda^{2s+2}_{-\square_{t}},u)+\frac{1}{2}(1+(-1)^{\beta})W_{k}^{*}
(\lambda^{2s+2}_{-\square_{j}},u)
& \text{\rm for} \ \alpha=+\square_{j},\ \ j\neq 0,1;\\
\sum_{t=0}^{k-1}W_{k}^{*}(\lambda^{2s+2}_{+\square_{t}},u)
& \text{\rm for} \ \alpha=\varnothing; \\
W_{k}^{*}(\lambda^{2s+2}_{\varnothing},u) & \text{\rm for} \ \alpha=
-\square_{j},\ \ j\neq 0,
\end{cases}
\end{equation}
in which,
\begin{equation}
\beta=
\begin{cases}
1 & W_{k}^{*}(\lambda^{2s+1}_{\alpha},u+1)=0,\\
0 & W_{k}^{*}(\lambda^{2s+1}_{\alpha},u+1)\neq 0.\\
\end{cases}
\end{equation}
We are now in position to specify the process $(X^{i})_{i=0}^{i=2N}$:\\
$\bullet$ $X^{0}=X^{2N}=\varnothing$ and $X^{i}$ is a shape with at
most $k-1$ rows\\
$\bullet$ for $1\leq i\leq N-1$, $(X^{2i+1}\setminus X^{2i},
X^{2i+2}\setminus
X^{2i})\in\{(-\square,\varnothing),(\varnothing,+\square),
(\varnothing,\varnothing),(+\square,-\square)\}$\\
$\bullet$ there does not exist any subsequence $(X^{2i}, X^{2i+1},
X^{2i+2})$ such that $(X^{2i+1}\setminus X^{2i},
X^{2i+2}\setminus X^{2i})=(+\square_1,-\square_{1})$\\
$\bullet$
the transition probabilities are given as follows\\
{\sf(1)} for $i=2\ell$, we obtain
\begin{equation*}
\mathbb{P}(X^{i+1}=\lambda^{i+1}_{\alpha}\vert X^{i}=\lambda^{i})=
\frac{V_{k}^{*}(\lambda^{i+1}_{\alpha},2N-i-1)}{W_{k}^{*}(\lambda^{i},2N-i)}.
\end{equation*}
{\sf(2)} for $i=2\ell+1$, we have
\begin{equation*}
\mathbb{P}(X^{i+1}=\lambda^{i+1}_{}\vert
X^{i}=\lambda^{i}_{\alpha})=
\frac{W_{k}^{*}(\lambda^{i+1}_{},2N-i-1)}
{V_{k}^{*}(\lambda^{i}_{\alpha},2N-i)}.
\end{equation*}
By construction,
\begin{equation}
\prod_{i=0}^{2N-1}\mathbb{P}(X^{i+1}=\lambda^{i+1}\mid
X^i=\lambda^i)= \frac{{ W_{k}^*}(\lambda^{2N}=\varnothing,0)} {{
W_{k}^*}(\lambda^{0}=\varnothing,2N)}
=\frac{1}{{W_{k}^*}(\varnothing,2N)},
\end{equation}
whence $(X^i)_{i=0}^{2N}$ generates random $2$-regular,
$k$-noncrossing partitions with uniform probability in $O(N)$ time
and space.\\
As for the computation of the transition probabilities, according to
Theorem \ref{T:algpk} the terms $\omega^{\nu}_{k, \ell}$ for $\nu\in
{W}_{k-1}$ and $1\leq \ell\leq 2N+1$ can be calculated in
$O(N^{k+1})$ time and $O(N^k)$
space.\\
We claim that for fixed indices $\nu=(\nu_1,\nu_2,\dots,\nu_{k-1})$
where $\nu\in{W}_{k-1}$ and $1\leq s\leq 2N$, $\sigma^{\nu;*}_{k,s}$
can be computed in $O(N)$ time. There are two cases: in case of
$s=2\ell_1$, using a {\bf For}-loop summing over the terms
$(-1)^{h}{\ell_1 \choose h} \omega^{\nu}_{k,2(\ell_1-h)+1}$, we
derive $\sigma^{\nu;*}_{k,2\ell_1}$. In case of $s=2\ell_1+1$, we
calculate $\sigma^{\nu-\mathbf{e_{i}};*}_{k,2\ell_1}$, where
$i\in\{0,1,\dots,k-1\}$ via eq.~(\ref{E:IE}). Then
$\sigma^{\nu;*}_{k,2\ell_1+1}$ follows in view of
$\sigma^{\nu;*}_{k,2\ell_1+1}=
\sum_{j=0}^{k-1}\sigma^{\nu-\mathbf{e_{j}};*}_{k,2\ell_1}$.\\
Furthermore, using $k$ nested {\bf For}-loops for $\nu_{i}$, we
derive $\sigma^{\nu;*}_{k,s}$ for arbitrary $\nu$ and $s$.
Consequently, we compute $\sigma^{\nu;*}_{k,s}$ for all $\nu$, $s$
with $O(N^k)+O(N^k\times N)=O(N^{k+1})$ time and $O(N^{k})$
space complexity.\\
Once the terms $\sigma^{\nu;*}_{k,s}$ for $\nu$ and $s$ are
calculated, we can compute the transition probabilities in $O(1)$
time. Therefore we obtain the transition probabilities
${W^{*}(\lambda^{i},2N-i)}$ in $O(N^{k+1})$ time and $O(N^{k})$
space complexity and the theorem follows.
\end{proof}
\NEW{\textbf{Remark} We remark here it is feasible to generate a
$2$-regular, $k$-noncrossing partition in $O(N^{k})$ time complexity
if we obtain the transition probabilities for
$\mathcal{B}_{W_{k-1}}^*$-walks follow the similar routine of
$\mathcal{P}_{W_{k-1}}$-walks. The reason we show a different
routine is we prefer to compute transition probabilities via the
relation between two different combinatorial objects.}
\section{Example: in case of $k=4$}\label{S:4}
In the following, we omit $k=4$ in the subscripts of the notations.
\subsection{Example: random $4$-noncrossing partitions}\label{S:p4}
Set $a^{i,j,r}_{s}$ denote the number of $\mathcal{P}_{Q_3}$-walks
of length $s$ starting at $(2,1,0)$, ending at $(i,j,r)\in Q_3$ and
set $f^{i,j,r}_{s}=a^{i,j,r}_{2s+1}$.

\begin{lemma}\label{L:p4}
{\rm (a)} Suppose  $1\leq \ell\leq N-1$ and $(i,j,r)\in W_3$, then
\begin{equation}\label{E:o4}
\omega^{i,j,r}_{s}=
\begin{cases}
f^{i,j,r}_{\ell}-f^{j,i,r}_{\ell}-f^{r,j,i}_{\ell}-f^{i,r,j}_{\ell}
+f^{j,r,i}_{\ell}+f^{r,i,j}_{\ell}
& \text{\rm for $s=2\ell+1$},\\
f^{i,j,r}_{\ell}+f^{i-1,j,r}_{\ell}+f^{i,j-1,r}_{\ell}+f^{i,j,r-1}_{\ell}
-f^{j,i,r}_{\ell}- f^{j-1,i,r}_{\ell}-
f^{j,i-1,r}_{\ell}-f^{j,i,r-1}_{\ell} &\\
-f^{r,j,i}_{\ell}-f^{r-1,j,i}_{\ell}-f^{r,j-1,i}_{\ell}-f^{r,j,i-1}_{\ell}
-f^{i,r,j}_{\ell}-f^{i-1,r,j}_{\ell}-f^{i,r-1,j}_{\ell}-f^{i,r,j-1}_{\ell}&\\
+f^{j,r,i}_{\ell}+f^{j-1,r,i}_{\ell}+f^{j,r-1,i}_{\ell}+f^{j,r,i-1}_{\ell}
+f^{r,i,j}_{\ell}+f^{r-1,i,j}_{\ell}+f^{r,i-1,j}_{\ell}+f^{r,i,j-1}_{\ell}
&  \text{\rm for $s=2\ell+2$}.\\
\end{cases}
\end{equation}
where for $a,b,c\in \mathbb{Z}$, $f^{a,b,c}_{\ell}=0$ for $a+b+c\geq
\ell+3$ and for $\ell=0$,
\begin{equation*}
f^{i,j,r}_{0}=
\begin{cases}
1 & \textrm{for } i=2,j=1,r=0; \\
0 & \textrm{otherwise.}
\end{cases}
\end{equation*}
{\rm (b)} $f^{p,q,s}_{\ell}$ satisfies the recursion
\begin{eqnarray*}
f^{p,q,s}_{\ell}=f^{p,q+1,s}_{\ell-1}&+&f^{p+1,q,s}_{\ell-1}+f^{p,q,s+1}_{\ell-1}
f^{p-1,q+1,s}
_{\ell-1}+f^{p+1,q-1,s}_{\ell-1}+f^{p-1,q,s+1}_{\ell-1}\\
&+&f^{p+1,q,s-1}_{\ell-1}+
f^{p,q-1,s+1}_{\ell-1}+f^{p,q+1,s-1}_{\ell-1}
+4f^{p,q}_{\ell-1}+f^{p-1,q,s}_{\ell-1}+f^{p,q-1,s}_{\ell-1}+f^{p,q,s-1}.
\end{eqnarray*}
\end{lemma}
\begin{proof}
We first prove assertion (a). Indeed, by definition of
$f^{i,j,r}_{\ell}$, we have
\begin{eqnarray}
a^{i,j,r}_{2\ell+1}= f^{i,j,r}_{\ell}  \quad \text{\rm and} \quad
a^{i,j,r}_{2\ell}
=f^{i,j,r}_{\ell-1}+f^{i-1,j,r}_{\ell-1}+f^{i,j-1,r}_{\ell-1}+
f^{i,j,r-1}_{\ell-1},
\end{eqnarray}
whence eq.~(\ref{E:o4}) follows from Proposition \ref{P:reflect} for
the case $k=4$. I.e.
\begin{equation}
\omega^{i,j,r}_{\ell}=a^{i,j,r}_{\ell}-a^{j,i,r}_{\ell}-
a^{r,j,i}_{\ell}-a^{i,r,j}_{\ell}
+a^{j,r,i}_{\ell}+a^{r,i,j}_{\ell}.
\end{equation}
Next, by construction, whence (b) and the proof of the lemma is
complete.
\end{proof}
Once $\omega^{i,j,r}_{\ell}$ for arbitrary $(i,j,r)\in W_{3}$ can be
calculated follow the routine given in Lemma \ref{L:p4}, we arrive
at the following algorithm:
\begin{corollary}\label{C:par4}
A random $4$-noncrossing partition can be generated, in $O(N^4)$
pre-processing time and $O(N^{4})$ space complexity, with uniform
probability in linear time. Each $4$-noncrossing partition is
generated with $O(N)$ space and time complexity, see Algorithm 4.
\end{corollary}
\begin{algorithm}\label{A:alg4}
\begin{algorithmic}[1]
\STATE{\it Tableaux} (Initialize the sequence of shapes,
$\{\lambda^{i}\}_{i=0}^{i=2N}$)

\STATE {$\lambda^{0}\leftarrow\varnothing$,
$\lambda^{2N}\leftarrow\varnothing$, $i\leftarrow 1$}

\WHILE {$i < 2N$ }

\IF {$i$ is even} \STATE  X[0]$\leftarrow$ ${\sf
V}(\lambda^{i+1}_{\varnothing}, 2N-(i+1))$

\STATE  X[1]$\leftarrow$ ${\sf V}(\lambda^{i+1}_{-\square_{1}},
2N-(i+1))$

\STATE  X[2]$\leftarrow$ ${\sf V}(\lambda^{i+1}_{-\square_{2}},
2N-(i+1))$

\STATE  X[3]$\leftarrow$ ${\sf V}(\lambda^{i+1}_{-\square_{3}},
2N-(i+1))$

\ENDIF

\IF {$i$ is odd}

\STATE  X[0]$\leftarrow$ ${\sf V}(\lambda^{i+1}_{\varnothing},
2N-(i+1))$

\STATE  X[1]$\leftarrow$ ${\sf V}(\lambda^{i+1}_{+\square_{1}},
2N-(i+1))$

\STATE  X[2]$\leftarrow$ ${\sf V}(\lambda^{i+1}_{+\square_{2}},
2N-(i+1))$

\STATE  X[3]$\leftarrow$ ${\sf V}(\lambda^{i+1}_{+\square_{3}},
2N-(i+1))$ \ENDIF

\STATE {\it sum} $\leftarrow$ X[0]+X[1]+X[2]+X[3]

\STATE {\it Shape} $\leftarrow$ {\sf Random}({\it sum}) ({\sf
Random} generates the random shape $\lambda^{i+1}_{+\square_{j}(or
-\square_{j})}$ with probability ${X[j]}/{\it sum}$ or
$\lambda^{i+1}_{\varnothing}$ with probability $X[0]/sum$)

\STATE $i \leftarrow i+1$

\STATE Insert {\it Shape} into {\it Tableau} (the sequence of
shapes).

\ENDWHILE

\STATE  {\sf Map}({\it Tableau}) (maps {\it Tableau} into its
corresponding $4$-noncrossing partition)
\end{algorithmic}
\caption{\small Uniform generation of $4$-noncrossing partitions}
\end{algorithm}
\subsection{Example: $2$-regular, random $4$-noncrossing partitions}
\label{S:p24}
\begin{lemma}\label{L:regq4}
The number of $\mathcal{B}_{W_{3}}^*$-walks ending at $\nu\in W_{3}$
of length $2\ell$, where $1\leq \ell\leq N$ is given by
\begin{equation}
\sigma^{i,j,r;*}_{2\ell} = \sum_{h}(-1)^{h}{\ell\choose
h}\omega^{i,j,r}_{2(\ell-h)+1}.
\end{equation}
Furthermore, we obtain the recurrence of
$\sigma^{i,j,r;*}_{2\ell-1}$ given by
\begin{equation}
\sigma^{i,j,r;*}_{2\ell-1}
=\sigma^{i,j,r;*}_{2\ell-2}+\sigma^{i-1,j,r;*}_{2\ell-2}
+\sigma^{i,j-1,r;*}_{2\ell-2}+\sigma^{i,j,r-1;*}_{2\ell-2}.
\end{equation}
\end{lemma}
\begin{corollary}\label{C:r4}
A random $2$-regular, $4$-noncrossing partition can be generated, in
$O(N^{5})$ pre-processing time and $O(N^{4})$ space complexity, with
uniform probability in linear time. Each $2$-regular,
$4$-noncrossing partition is generated with $O(N)$ space and time
complexity, see Algorithm 5, 6.
\end{corollary}
\begin{algorithm}\label{A:algr4}
\begin{algorithmic}[1]
\STATE{\it Tableaux} (Initialize the sequence of shapes to be a list
$\{\lambda^{i}\}_{i=0}^{i=2N}$)

\STATE {$\lambda^{0}\leftarrow\varnothing$,
$\lambda^{2N}\leftarrow\varnothing$,
$\lambda^{2N-1}\leftarrow\varnothing$, $i\leftarrow 1$}

\WHILE {$i < 2N-1$ }

\IF {$i$ is even}

\STATE  X[0]$\leftarrow$ ${\sf V^{*}}(\lambda^{i+1}_{\varnothing},
2N-(i+1))$

\STATE  X[1]$\leftarrow$ ${\sf V^{*}}(\lambda^{i+1}_{+\square_{1}},
2N-(i+1))$, X[2]$\leftarrow$ ${\sf
V^{*}}(\lambda^{i+1}_{-\square_{1}}, 2N-(i+1))$

\STATE  X[3]$\leftarrow$ ${\sf V^{*}}(\lambda^{i+1}_{+\square_{2}},
2N-(i+1))$, X[4]$\leftarrow$ ${\sf
V^{*}}(\lambda^{i+1}_{-\square_{2}}, 2N-(i+1))$

\STATE  X[5]$\leftarrow$ ${\sf V^{*}}(\lambda^{i+1}_{+\square_{3}},
2N-(i+1))$, X[6]$\leftarrow$ ${\sf
V^{*}}(\lambda^{i+1}_{-\square_{3}}, 2N-(i+1))$

\ELSE

\STATE  X[0]$\leftarrow$ ${\sf W^{*}}(\lambda^{i+1}_{\varnothing},
2N-(i+1))$

\STATE X[1]$\leftarrow$ ${\sf W^*}(\lambda^{i+1}_{+\square_{1}},
2N-(i+1))$, X[2]$\leftarrow$ ${\sf
W^{*}}(\lambda^{i+1}_{-\square_{1}}, 2N-(i+1))$

\STATE  X[3]$\leftarrow$ ${\sf W^{*}}(\lambda^{i+1}_{+\square_{2}},
2N-(i+1))$, X[4]$\leftarrow$ ${\sf
W^{*}}(\lambda^{i+1}_{-\square_{2}}, 2N-(i+1))$

\STATE  X[5]$\leftarrow$ ${\sf W^{*}}(\lambda^{i+1}_{+\square_{3}},
2N-(i+1))$, X[6]$\leftarrow$ ${\sf
W^{*}}(\lambda^{i+1}_{-\square_{3}}, 2N-(i+1))$

\STATE {\sf FLAG}(flag0, flag1, flag2, flag3)

\STATE {\it sum} $\leftarrow$ X[0]+X[1]+X[2]+X[3]+X[4]+X[5]+X[6]

\ENDIF

\STATE {\it Shape} $\leftarrow$ {\sf Random}({\it sum}) ({\sf
Random} generates the random shape $\lambda^{i+1}_{+\square_{j}(or
-\square_{j})}$ with probability {X[2j-1]}$/{\it sum}$ (or
X[2j]$/sum$) or $\lambda^{i+1}_{\varnothing}$ with probability
X[0]$/sum$)

\STATE {flag0 $\leftarrow$ 1, flag1 $\leftarrow$ 1, flag2
$\leftarrow$ 1, flag3 $\leftarrow$ 1}

\IF {$i$ is even and Shape=$\lambda^{i+1}_{+\square_{1}}$}

\STATE flag1 $\leftarrow$ 0

\ENDIF

\IF {$i$ is even and Shape=$\lambda^{i+1}_{-\square_{1}}$}

\STATE flag2 $\leftarrow$ 0

\ENDIF

\IF {$i$ is even and Shape=$\lambda^{i+1}_{-\square_{2}}$}

\STATE flag2 $\leftarrow$ 0

\ENDIF

\IF {$i$ is even and Shape=$\lambda^{i+1}_{-\square_{3}}$}

\STATE flag2 $\leftarrow$ 0

\ENDIF

\IF {$i$ is even and Shape=$\lambda^{i+1}_{+\square_{2}}$}

\STATE flag3 $\leftarrow$ 0

\ENDIF

\IF {$i$ is even and Shape=$\lambda^{i+1}_{+\square_{3}}$}

\STATE flag3 $\leftarrow$ 0

\ENDIF

\IF {$i$ is even and Shape=$\lambda^{i+1}_{\varnothing}$}

\STATE flag0 $\leftarrow$ 0

\ENDIF

\STATE Insert {\it Shape} into {\it Tableaux} (the sequence of
shapes)

\STATE $i \leftarrow i+1$

\ENDWHILE

\STATE  {\sf Map}({\it Tableaux}) (maps {\it Tableaux} into its
corresponding $2$-regular, $4$-noncrossing partition)
\end{algorithmic}
\caption{\small Uniform generation of $2$-regular, $4$-noncrossing
partitions.}
\end{algorithm}

\begin{algorithm}
\label{A:alg5}
\begin{algorithmic}[1]
\STATE{\it FLAG}(flag0, flag1, flag2, flag3)

\IF {flag0 =0}

\STATE {X[2], X[4], X[6]$\leftarrow$ 0}

\ENDIF

\IF {flag1 =0}

\STATE {X[0], X[1], X[2], X[3], X[5]$\leftarrow$ 0}

\ENDIF

\IF {flag2 =0}

\STATE {X[1], X[2], X[3], X[4], X[5], X[6]$\leftarrow$ 0}

\ENDIF

\IF {flag3 =0}

\STATE {X[1], X[3], X[5]$\leftarrow$ 0}

\ENDIF

\end{algorithmic}
\caption{\small {\sf FLAG}: Distinguish the last odd step}
\end{algorithm}
\bibliographystyle{plain}
\bibliography{unik}
\end{document}